\documentclass[letter, 12pt]{amsart}

\usepackage{amscd} \usepackage{amsmath} \usepackage{amssymb}
\usepackage{amsthm} \usepackage{arydshln} \usepackage{fancyhdr}
\usepackage{mathtools} \usepackage{verbatim} \usepackage[usenames,
dvipsnames]{color} \usepackage{tikz} \usetikzlibrary{snakes}
\usepackage[all]{xy} \usepackage[colorlinks=true, linkcolor=blue,
citecolor=blue, pagebackref=true]{hyperref}
\usepackage[pagebackref=true]{hyperref}
\usepackage{mathrsfs}
\renewcommand{\epsilon}{\varepsilon} \renewcommand{\phi}{\varphi}
\newcommand{\cA}{\mathcal{A}} \newcommand{\cH}{\mathcal{H}}
\newcommand{\calP}{\mathcal{P}} \newcommand{\NN}{\mathbb{N}}
\newcommand{\RR}{\mathbb{R}} 
 \newcommand{\bi}{\boldsymbol i}
\newcommand{\bk}{\boldsymbol k} \newcommand{\bj}{\boldsymbol j}
\newcommand{\bq}{\boldsymbol q} \newcommand{\ba}{\boldsymbol a}

\theoremstyle{plain} \newtheorem{theorem}{Theorem}

\theoremstyle{definition} \newtheorem*{example*}{Example}

\theoremstyle{plain} \newtheorem{lemma}[theorem]{Lemma}
\newtheorem{corollary}[theorem]{Corollary}

\theoremstyle{remark} \newtheorem*{remark*}{Remark}

\providecommand{\floor}[1]{\lfloor#1\rfloor}

\providecommand{\abs}[1]{\lvert#1\rvert}

\title[Recurrence to shrinking targets]{Recurrence to shrinking
  targets on typical self-affine fractals} \author[H.~Koivusalo \and
F.~A.~Ram{\'i}rez]{Henna Koivusalo \and Felipe A.~Ram{\'i}rez}
\address{Department of Mathematics, University of
  York, Heslington, YO10 5DD, UK} \email{henna.koivusalo@york.ac.uk}
\address{Wesleyan University, Middletown, CT, USA}
\email{framirez@wesleyan.edu}

\date{\today}

\begin{document}

\thispagestyle{empty}

% ===============================================

\begin{abstract}
  We explore the problem of finding the Hausdorff dimension of the set
  of points that recur to shrinking targets on a self-affine
  fractal. To be exact, we study the dimension of a certain related
  symbolic recurrence set. In many cases this set is equivalent to the
  recurring set on the fractal.
\end{abstract}

% ===============================================

\thanks{H.K.~was supported by EPSRC Programme Grant
  EP/L001462 and F.A.R.~by EP/J018260/1.}

\maketitle

% \tableofcontents
% {\footnotesize\tableofcontents}

\section{Introduction}
\label{sec:intro}

Given a dynamical system $(X, T)$, and a sequence $\mathcal B=(B_k)$
of subsets of $X$, a {\bf shrinking target problem} is the study of
the size of the recurring set
\begin{equation*}
  R(\mathcal B) = \left\{x\in X : T^k(x)\in B_k \text{ for infinitely many }
    k\in\mathbb{N}\right\}.
\end{equation*}
Viewing the sets $B_k$ as ``targets,'' the set $R(\mathcal B)$ is the
set of points whose orbits under the dynamics hit the targets
infinitely often.

Interest in recurrence is classical in the study of dynamical
systems. There are many ways to interpret the ``size'' of
$R(\mathcal B)$. One way is with respect to an invariant measure on
$X$. In this context, for example, for any positive measure set
$B\subset X$, the Poincar\'e recurrence theorem implies that almost
all points of $X$ return to $B$ infinitely often, so that
$R(\mathcal B)$ has full measure when $B_k = B$ for all $k$. It is
thus fairly natural to ask what happens when the targets shrink. In
many settings $R(\mathcal B)$ has either full or zero measure
depending on the rate of shrinking of $B_k$. See, for example, Chernov
and Kleinbock \cite{ChernovKleinbock01} for a general treatment of
such a problem.

For exponentially decreasing balls on Julia sets, Hill and Velani
calculated first the Hausdorff dimension \cite{HillVelani95} and then
the Hausdorff measure \cite{HillVelani02} of the recurring set, in the
latter case proving a $0$-$1$-law for $R(\mathcal B)$. Further, in
\cite{HillVelani99} they considered the dimension of $R(\mathcal B)$
for toral automorphisms. More recently, the problem has been studied
for infinitely branching Markov maps \cite{Reeve}, piecewise expanding
maps of the unit interval \cite{PerssonRams} and in the dynamical
system of continued fraction expansions \cite{BingWangWuXu14}. In all
of these cases the dimension of the recurring set is given as a zero
point of a generalized pressure functional.

Recurring points are not only an interesting object in abstract
settings; they also appear in number theory in the study of
Diophantine approximation. The doubly-metric inhomogeneous version of
Khintchine's Theorem, usually stated as a $0$-$1$-law for the set of
``approximable'' pairs in $\RR^d\times\RR^d$, can be
interpreted through Fubini's Theorem as a $0$-$1$-law for the
recurring set of shrinking targets on a torus, which holds for almost
any toral translation. This connection to toral translations is
made pleasantly visible by Fayad \cite{Fayad06}, where among other
things he gives a proof of a result of Kurzweil \cite{Kurzweil66} from
this point-of-view. Similarly, one can interpret the classical
Jarn\'ik-Besicovitch theorem as a dimension result for a recurring set
of shrinking targets on the torus. For a more recent example, see
Bugeaud, Harrap, Kristensen and Velani
\cite{BugeaudHarrapKristensenVelani10}.

The above results mainly focus on conformal dynamical systems. In this
note we study the shrinking target problem in a simple non-conformal
case, namely, under affine dynamics. In particular, we will calculate
the dimension of the recurring set on a certain fractal set, the
definition of which will be given next.

Hutchinson \cite{Hutchinson81} originated the systematic study of
iterated function systems (IFSs) $\{f_1, \dots, f_m\}$, which define a
fractal set through the identity $E= \cup_{i=1}^m f_i(E)$. In
particular, the {\bf self-affine fractal} is the unique, nonempty,
compact set $E$ satisfying
\[
  E_{(a_1,\dots, a_m)} =E = \bigcup_{i=1}^m T_i(E)+a_i =
  \bigcup_{i=1}^m f_i(E),
\]
generated by iterating the affine maps
$\{f_i=T_i + a_i\}_{i=1,\dots, m}$. Here $\{T_1,\dots,T_m\}$ is a
collection of bijective linear contractions on $\mathbb R^d$ and
$a_1,\dots, a_m\in \mathbb R^d$ are translation vectors. If the images
$f_i(E)$ are disjoint it is possible to define a mapping $F:E \to E$
by setting
\[
  F(x)= f_i^{-1}(x)
\]
for points $x\in f_i(E)$, for all $i=1,\dots, m$. This sets up a
dynamical system $(E, F)$, and so it makes sense to formulate a
shrinking target problem on $E$.

The study of dimensional properties of self-affine sets comes in three
branches: First, as self-affine carpets due to, originally, Bedford
\cite{Bedford84} and McMullen \cite{McMullen84}; second, for generic
self-affine sets in the footsteps of Falconer \cite{Falconer88};
third, for \emph{every} self-affine set in some family, as
in~\cite{Falconer92}. In this paper we consider a shrinking target
problem associated to the set $E$ in the generic sense, that is, for
generic translations $\ba =(a_1,\dots, a_m)\in\RR^{dm}$. A
complementary problem of mass escape for carpet type sets has been
studied by Ferguson, Jordan and Rams in \cite{FergusonJordanRams}.

The main result of this article can be stated informally as follows.
\theoremstyle{plain} \newtheorem*{main}{Main Theorem}
\begin{main}[An informal version of Theorem \ref{thm}]
  The dimension of the recurring set to shrinking targets on
  self-affine fractals is almost surely the unique zero point
  $s_0\in\RR^+$ of a modified pressure function $P(s)$.
\end{main}
We will impose a quasimultiplicativity condition \eqref{eq:cone} on
the maps $T_i$, and give the modified pressure
formula~\eqref{eq:pressure}. In this article, the target sets $B_k$
are dynamical balls corresponding to cylinder sets in the symbolic
space (for definitions, see~\S\ref{subsec:self-affine}). These target
sets match the dynamics on the set. The method of proof is largely
classical and the main difficulty lies in finding the right definition
for the pressure functional. Also, in order to obtain the lower bound
we need to construct a Cantor type subset of the recurring set. We do
this in the spirit of the proof of the mass transference principle of
Beresnevich and Velani \cite{BeresnevichVelani06}. Their result cannot
be applied to the affine setting directly.

It is classical to study the size of the recurring set in terms of
invariant measures, and it turns out that, at least in certain cases,
this problem is easily solved based on \cite{ChernovKleinbock01}. In
particular, in \S\ref{sec:examples} we notice that as a corollary of
\cite[Theorem 2.1]{ChernovKleinbock01} it is possible to measure the
size of the recurring set in terms of the Gibbs measure at the
dimension of the self-affine set $E$. According to
\cite{KaenmakiReeve} Gibbs measures are measures of maximal dimension
on $E$, which in part motivates their use in this context. For further
reading on general thermodynamic formalism and Gibbs measures we refer
to \cite{Bowen75}, and in relation to self-affine sets, to
\cite{KaenmakiReeve}).
\theoremstyle{plain}\newtheorem*{coro}{Corollary to
  \cite{ChernovKleinbock01}}
\begin{coro}[An informal version of Corollary \ref{cor:Borel}]
  In terms of the Gibbs measure at the dimension of the self-affine
  set $E$, either almost all or almost no points are recurring,
  depending on whether the sum of the measures of the target sets
  diverges or converges.
\end{coro}
Even though these formulations of the shrinking target problem give a
very natural starting point to the analysis of shrinking targets for
nonconformal dynamics, some questions remain. We discuss these at the
end of ~\S\ref{sec:examples}.

\subsection*{Acknowledgments} We are grateful to Thomas Jordan for
suggesting the problem to the first author, and for helpful
discussions. We also thank Antti K{\"a}enm{\"a}ki for taking the time
to read our manuscript, and for giving valuable feedback.

\section{Shrinking targets, fractals, and symbolic dynamics}

Let $E_{\ba}$, $f_i$, $T_i$ be as in the introduction.

\subsection{Self-affine fractals and symbolic
  dynamics}\label{subsec:self-affine}

It is standard to associate to $E_{\ba}$ a symbolic dynamical system
given by the left shift $\sigma$ on the sequence space
$\mathcal{A}^\NN = \{1, \dots, m\}^\NN$. One identifies points
$\bi=(i_1, i_2, \dots)\in\mathcal{A}^\NN$ with points on $E_{\ba}$
through the projection
\begin{equation*}
  \pi(\bi) = \lim_{k\to\infty}f_{i_1}\circ \dots \circ f_{i_k} (0) \in E_{\ba},
\end{equation*}
and if this projection is injective, then $\sigma$ induces a
well-defined dynamical system on $E_{\ba}$, taking the point
$\pi(\bi)=p\in E_{\ba}$ to $f_{i_1}^{-1}(p)$. In this case, the
shrinking targets problem (defined in the next section) on the
symbolic dynamical system corresponds to a shrinking targets problem
on the fractal with the dynamics induced by $\sigma$.

\subsection{Shrinking targets on a symbolic dynamical system}

A natural way to define a shrinking targets problem on the dynamical
system $\sigma:\cA^\NN\to\cA^\NN$ is to select a target point
$\bj\in\cA^\NN$ and define the targets to be cylinder sets
\begin{equation*}
  [\bj|_k] = \left\{\bi\in\cA^\NN : \bi|_k = \bj|_k\right\}\subset\cA^\NN
\end{equation*}
of increasing length $k\in\NN$ (and hence decreasing diameter) around
that point, where $\bj|_k$ and $\bi|_k$ denotes the truncations $(j_1,
\dots, j_k)$ and $(i_1, \dots, i_k)$ of $\bj$ and $\bi$, respectively.

In this paper we treat the shrinking targets problem for $\mathcal B=
\{B_k\}_{k\in\NN}$ where $B_k = [\bj|_{\ell_k}]$ and
$\{\ell_k\}\subseteq\NN$ is some fixed (non-decreasing) sequence. We
denote the recurring set $R(\mathcal B)$ by $R(\bj)$. It is exactly
\begin{align*}
  R(\bj) &= \left\{ \bi\in\cA^\NN : \sigma^k(\bi)|_{\ell_k} = \bj|_{\ell_k} \textrm{ for infinitely many } k\in\NN\right\} \\
  &= \limsup_{k\to\infty} R(\bj, k)
\end{align*}
where $R(\bj, k) = \left\{\bi\in\cA^\NN : \sigma^k(\bi)|_{\ell_k} =
  \bj|_{\ell_k}\right\}$.

\subsection{Shrinking targets on self-affine fractals}\label{sec:on fractal}

We project our symbolic shrinking targets problem to the fractal
$E_{\ba}$ through $\pi$. That is, we seek the Hausdorff dimension of
the set
\begin{equation*}
  \widetilde R(\pi(\bj)) := \pi(R(\bj)) \subset E_{\ba}.
\end{equation*}
An intuitive interpretation is to think of the shrinking targets
around $\pi(\bj)$ as the succession of ``construction sets'' of
$E_{\ba}$ to which $\pi(\bj)$ belongs. (One may think of the stardard
Cantor set, for example, and take the shrinking targets to be
subintervals appearing in its standard construction.) Though these are
very natural shrinking targets on a fractal, the notion is really only
meaningful if the affine maps $\{f_i\}$ satisfy a strong enough
separation condition. Still, the set $\widetilde R(\pi(\bj))$ is
well-defined and can be studied regardless.

\section{List of notations}

This list of standard notations may be used as reference.

\begin{itemize}
\item $\cA^\NN = \{1,\dots,m\}^{\mathbb N}$ is the sequence space on
  $m$ letters.
\item $\boldsymbol{j}|_{ k} = j_1 j_2 \dots j_k$ is the string
  obtained by truncating $\bj\in\cA^\NN$.
\item $\abs{\boldsymbol{i}}$ is the length of a finite word
  $\boldsymbol{i}=(i_1,\dots, i_k)$, so for example
  $\abs{\boldsymbol{j}|_{k}}=k$.
\item $\boldsymbol{i}\boldsymbol{j}|_{k} = i_1 i_2 \dots
  i_{\abs{\boldsymbol{i}}} j_1 j_2 \dots j_k$ is a concatenation of
  finite words.
\item $\bi\wedge\bi'$ is the maximal common beginning of $\bi,
  \bi'\in\cA^\NN$.
\item $[\bq]=\{\bi\in \cA^\NN : \bi|_{\abs{\bq}}=\bq\}$ is the
  cylinder set defined by the finite word $\bq$.
\item Let $f_i:\mathbb R^d\to\mathbb R^d$ be the affine maps
  $f_i(x)=T_i(x)+ a_i$.
\item For $\abs{\bi}=k$, denote $f_{\bi}=f_{i_1}\circ\dots\circ
  f_{i_k}$, and similarly for $T_{\bi}$.
\item $R(\bj)=\{\bi\in \cA^\NN :
  \sigma^k(\bi)|_{\ell_k}=\bj|_{\ell_k}\textrm{ infinitely often}\}$
  is our symbolic recurring set.
\item $\widetilde R(\pi(\bj)) = \pi(R(\bj))$.
\item $R(k)=R(\bj,k)= \{\bi\in \cA^\NN :
  \sigma^k(\bi)|_{\ell_k}=\bj|_{\ell_k}\}$, so that we have
  $R(\bj)=\limsup R(k)$.
\end{itemize}

\section{The singular value function}

For a linear map $T:\mathbb R^d\to \mathbb R^d$ define the singular
values
\begin{equation*}
  \sigma_1(T)\ge \dots \ge \sigma_d(T)
\end{equation*}
as the lengths of the semiaxes of the ellipsoid $T(B(0,1))$. For any
$0\le t\le d$, let
\[
\phi^t(T)=\sigma_1(T)\dots
\sigma_{\floor{t}}(T)\sigma_{\floor{t}+1}(T)^{\{t\}}
\]
where $\floor{t}$ is the integer part of $t$ and $\{t\}$ is the
fractional part.  The function $\phi^t(T)$ is called the singular
value function, and it satisfies
\begin{equation}\label{eq:sing1}
  \sigma_d(T)^t\le \phi^t(T)\le \sigma_1(T)^t.
\end{equation}
Furthermore, if $T_1, \dots, T_n$ are linear maps on $\RR^d$, then
\begin{equation}\label{eq:sing2}
  \min_{i=1,\dots, n}\sigma_d(T_i)^\delta \le \frac{\sum_{i=1}^n \phi^{t+\delta}(T_i)}{\sum_{i=1}^n\phi^t(T_i)}\le \max_{i=1,\dots, n}\sigma_1(T_i)^\delta
\end{equation}
for all sufficiently small $\delta>0$. The singular value function is
submultiplicative, that is, if also $U: \mathbb R^d\to\mathbb R^d$ is
a linear map, then $\phi^t(TU)\le \phi^t(T)\phi^t(U)$.  For proofs see
\cite{Falconer88}, for example.

We study collections of linear maps for which there exists some
constant $D\in(0,1)$ such that
\begin{equation*}
  \phi^t(T_{\bi\bi'})\ge D\phi^t(T_{\bi})\phi^t(T_{\bi'})
\end{equation*}
for any finite words $\bi, \bi'$.  This condition holds, for example,
for similarities and for maps on $\RR^2$ satisfying the \emph{cone
  condition}, meaning that there is a cone that is mapped strictly
into itself by all the $T_i$'s, see K\"aenm\"aki and
Shmerkin~\cite{KaenmakiShmerkin09}.

\section{The main result and its proof}

Given the affine maps $\{f_i\}_{i=1, \dots, m}$, a point
$\bj\in\cA^\NN$, and $s\in\RR^+$, we define
\begin{equation}
  P(s,\boldsymbol{j}) :=
  \lim_{k\to\infty}\frac{1}{k}\log\sum_{\abs{\boldsymbol{i}}=k}\varphi^s(T_{\boldsymbol{i}\boldsymbol{j}|_{
      \ell_k}}).\label{eq:pressure}
\end{equation}
This is a modification of the standard pressure function (see
\cite{Falconer88} for example), which is defined to be
\begin{equation}\label{eq:ordinary pressure}
  \mathcal P(s)=\lim_{k\to \infty}\frac{1}{k}\log \sum_{\abs{\bi}=k}\phi^s(T_{\bi}).
\end{equation}
In fact, notice\footnote{We thank an anonymous referee for making this
  observation.} that in the presence of~\eqref{eq:cone} we have
\begin{equation}
  P(s,\boldsymbol{j}) = \mathcal P(s) + 
  \lim_{k\to\infty}\frac{1}{k}\log \varphi^s(T_{\boldsymbol{j}|_{
      \ell_k}}).\label{eq:refpressure}
\end{equation}
Notice also that it is possible for this limit to not exist. However,
if it \emph{does} exist for some $\bj$ and $s\in\RR^+$, then in fact
$P(s, \bj)$ is a continuous and non-increasing function of $s$. In
this case, we prove that the Hausdorff dimension of the recurring set
$\widetilde R(\pi(\bj))$ is almost surely the unique zero point
$s_0\in\RR^+$ of $P(\cdot, \bj)$.

\begin{theorem}\label{thm}
  Let $E_{\boldsymbol{a}}$ be the fractal determined by the affine
  maps $\{T_i + a_i \}_{i=1}^m$, and suppose there is some $D\in(0,1)$
  such that
  \begin{equation}\label{eq:cone}
    \phi^t(T_{\bi\bi'})\ge D\phi^t(T_{\bi})\phi^t(T_{\bi'})
  \end{equation}
  for any finite words $\bi, \bi'$. Furthermore, assume that
  $\max_i \sigma_1(T_i) < \frac{1}{2}$. Consider shrinking targets
  defined by cylinder sets of length
  $\{\ell_k \}_{k\in\NN}\subseteq\NN$ around some fixed
  $\bj\in\cA^\NN$.

  If $\lim_{k\to\infty}\frac{1}{k}\log \varphi^s(T_{\bj |_{\ell_k}})$
  exists in $[-\infty, 0]$, then with respect to any ergodic
  $\sigma$-invariant probability measure there is almost surely a
  unique number $s_0\in\RR^+$ satisfying $P(s_0, \bj) = 0$, and for
  almost every $\ba\in\RR^{dm}$, the Hausdorff dimension of the
  recurring set $\widetilde{R}(\pi(\bj))\subset E_{\ba}$ equals $s_0$.
\end{theorem}

\begin{remark*}
  If the maps $\{T_i + a_i\}_{i=1}^m$ satisfy the strong separation
  condition then the set $\widetilde R(\ba,\bj)$ from
  Theorem~\ref{thm} has a dynamical meaning as a subset of the fractal
  $E_{\ba}$.
\end{remark*}

Of course, it now becomes a question of \emph{when}
$\lim_{k\to\infty}\frac{1}{k}\log \varphi^s(T_{\bj |_{\ell_k}})$
exists. In the following corollary we show that in particular it
exists if $\lim_{k\to\infty}\ell_k/k$ does.

\begin{corollary}
  With the assumptions and notation of Theorem~\ref{thm}, if
  $\lim_{k\to\infty}\ell_k/k$ exists, then Theorem~\ref{thm} applies
  and we will have
  \begin{equation}
    \label{eq:cor}
    \dim_{\cH}\widetilde{R}(\pi(\bj)) =
    \begin{cases}
      \dim_{\cH}E_{\ba} &\textrm{if } \lim_{k\to\infty} \ell_k/k =0 \\
      s_0 &\textrm{if } \lim_{k\to\infty} \ell_k/k \in (0,\infty)\\
      0 &\textrm{if } \lim_{k\to\infty} \ell_k/k = \infty
    \end{cases}
  \end{equation}
  almost surely (in the sense of the theorem).
\end{corollary}

\begin{proof}
  If $\lim_{k\to\infty}\ell_k/k = \infty$, then the result follows
  from Lemma~\ref{lem:superlinear}, and in fact holds for \emph{every}
  target point $\bj$ and translation $\ba$.

  If $\lim_{k\to\infty}\ell_k/k\in [0,\infty)$, then
  Lemma~\ref{lem:number exists} implies the existence of
  $\lim_{k\to\infty}\frac{1}{k}\log \varphi^s(T_{\bj |_{\ell_k}})$, so
  Theorem~\ref{thm} applies. It is only left to observe that in the
  case where $\ell_k = o(k)$, we will have $\calP(s) = P(s,\bj)$, so
  the unique zero points coincide, and therefore the recurring set has
  the same Hausdorff dimension as the fractal.
\end{proof}

% \theoremstyle{plain} \newtheorem*{speculation}{Speculative Theorem}
% \begin{speculation}
%   Can always find a system with cone condition whose pressure (our
%   version) is super close to the pressure (our version) of the
%   original\dots Shmerkin etc
% \end{speculation}

\subsection{Existence and uniqueness}
\label{sec:existence}

The next lemma shows that if there is some $L\in[0,\infty)$ for which
$\ell_k \sim L\cdot k$, then there exists a unique number
$s_0\in\RR^+$ such that $P(s_0,\bj)=0$. It is furthermore almost
surely independent of $\bj$ with respect to any $\sigma$-invariant
measure.

\begin{lemma} \label{lem:number exists} With the same assumptions and
  notation as in Theorem~\ref{thm}, suppose that
  $\lim_{k\to\infty}\ell_k/k \in [0,\infty)$. Then, with respect to
  any $\sigma$-ergodic probability measure on $\cA^\NN$, there almost
  surely exists a number $s_0\in\RR^+$ satisfying $P(s_0, \bj)=0$.
  Furthermore, this number is unique and almost surely independent of
  $\bj\in\cA^\NN$.
\end{lemma}

\begin{proof}
  Suppose $\ell_k\sim L\cdot k$, $L>0$. Fix $s$. Let
  \begin{equation*}
    X(m,n) = \log \varphi^s(T_{\sigma^m(\bj)\mid_{n-m}}).
  \end{equation*}
  We will apply Kingman's subadditive ergodic theorem to the sequence
  $X(m,n)$ with respect to the transformation $\sigma$, so we must
  check that $X$ is subadditive and $\sigma$-equivariant. Clearly all
  $X(m,n)$ are integrable, since they only depend on finite words.

  \textbf{Equivariance} is verified by the following routine
  calculation:
  \begin{multline*}
    X(m,n)\circ \sigma = \log \phi^s(T_{\sigma^m(\sigma(\bj))|_{n-m}})
    \\
    = \log \phi^s(T_{\sigma^{(m+1)}(\bj)|_{(n+1)-(m+1)}}) = X(m+1,
    n+1).
  \end{multline*}
  For \textbf{subadditivity} we must check that
  $X(0, n) \leq X(0, m) + X(m, n)$. Again, this is routine:
  \begin{multline*}
    X(0, m) + X(m, n) = \log \phi^s(T_{\bj|_m}) + \log
    \phi^s(T_{\sigma^m(\bj)|_{n-m}}) \\
    =\log \phi^s(T_{\bj|_m})\phi^s(T_{\sigma^m(\bj)|_{n-m}})\geq \log
    \phi^s(T_{\bj|_n}) = X(0,n).
  \end{multline*}

  Therefore, by the subadditive ergodic theorem, there almost surely
  exists a limit
  \[
    \lim_{n\to\infty} \frac{X(0, n)}{n} = \lim_{k\to\infty} \frac{X(0,
      \ell_k)}{\ell_k}
  \]
  which is $\sigma$-invariant and so almost everywhere constant in
  $\bj$ with respect to any $\sigma$-ergodic measure. Hence,
  \begin{equation*}
    \lim_{k\to\infty}\frac{X(0,\ell_k)}{k} =
    \lim_{k\to\infty}\frac{X(0,\ell_k)}{\ell_k}\cdot\frac{\ell_k}{k} = L\lim_{k\to\infty} \frac{X(0, \ell_k)}{\ell_k} 
  \end{equation*}
  also exists and is almost surely constant in $\bj$, and similarly
  for
  \begin{equation*}
    P(s, \bj) = \calP(s) + \lim_{k\to\infty}\frac{X(0,\ell_k)}{k}. 
  \end{equation*}

  Notice that the limit is continuous and strictly decreasing in $s$
  since by \eqref{eq:sing2},
  \[
    \sigma_-^{(k+\ell_k)\delta}\le
    \frac{\sum_{\abs{\bi}=k}\phi^{s+\delta}(T_{\bi\bj|_{\ell_k}})}{\sum_{\abs{\bi}=k}\phi^s(T_{\bi\bj|_{\ell_k}})}\le
    \sigma_+^{(k+\ell_k)\delta},
  \]
  where $0<\sigma_-= \min_i \sigma_d(T_i)$ and
  $\sigma_+= \max_i\sigma_1(T_i)<1$.

  Since $P(0,\bj)>0 $ and $P(s,\bj)\to -\infty$ when $s\to\infty$, the
  existence of a unique zero follows.
\end{proof}

\subsection{Upper bound}

The following lemma shows that $s_0$ is an upper bound for the
dimension of the recurring set $\widetilde R(\pi(\bj))$.

\begin{lemma}\label{lem:upper}
  Let $\{\ell_k\}\subseteq\NN$ be a non-decreasing divergent sequence
  and suppose that $P(\cdot,\bj)$ is decreasing, with unique zero
  $s_0\in\RR$. Then the Hausdorff dimension of
  $\widetilde R(\pi(\bj))$ is bounded above by $s_0$.
\end{lemma}

\begin{proof} 
  We will show that for any $t > s_0$, we have
  $\mathcal{H}^t(\widetilde R(\pi(\bj)) = 0$, where $\mathcal{H}^t$
  denotes Hausdorff measure.

  Notice that when $t>s_0$, then for all large $k$,
  \[
    \sum_{\abs{\bi}=k}\phi^{t}(T_{\bi\bj|_{\ell_k}}) \le
    \exp(-k\epsilon)
  \]
  for some $\epsilon >0$. Therefore
  \begin{equation}
    \sum_{k=0}^\infty \sum_{\abs{\bi}=k}\phi^{t}(T_{\bi\bj|_{\ell_k}})
    <\infty.\label{eq:upperconv}
  \end{equation}

  For all $N$, we have that
  \[
    R(\bj)=\limsup_{k\to\infty} \left\{\bi\in \cA^\NN\mid
      \sigma^k(\bi)|_{\ell_k}=\bj|_{\ell_k}\right\}\subset \bigcup
    _{k=N} ^\infty \bigcup_{\abs{\bi}=k}[\bi\bj|_{\ell_k}],
  \]
  hence,
  \[
    \widetilde R(\pi(\bj)) \subset \bigcup _{k=N} ^\infty
    \bigcup_{\abs{\bi}=k}\pi[\bi\bj|_{\ell_k}].
  \]
  Since the $T_i$'s are contractions, there is some $M>0$ sufficiently
  large that $f_i(B(0,M)) \subset B(0,M)$ for all $i=1,\dots,m$. Then
  we have the inclusion
  $\pi[\bi\bj|_{\ell_k}]\subset f_{\bi\bj|_{\ell_k}}(B(0,M))$, and
  therefore
  \[
    \widetilde R(\pi(\bj)) \subset \bigcup _{k=N} ^\infty
    \bigcup_{\abs{\bi}=k}f_{\bi\bj|_{\ell_k}}(B(0,M)).
  \]

  Recall that
  $ \sigma_1(T_{\bi\bj|_{\ell_k}}), \dots,
  \sigma_d(T_{\bi\bj|_{\ell_k}})$
  are the lengths of the semiaxes of the ellipsoid
  $T_{\bi\bj|_{\ell_k}}(B(0,1))$, in decreasing order. Therefore, the
  number of cubes of side-length
  $\sigma_{\floor{t}+1}q(T_{\bi\bj|_{\ell_k}})$ required to cover the
  sets $f_{\bi\bj|_{\ell_k}}(B(0,M))$ is bounded by
  \begin{equation*}
    A\sigma_1(T_{\bi\bj|_{\ell_k}})\dots \sigma_{\floor{t}+1}(T_{\bi\bj|_{\ell_k}})^{-\floor{t}},
  \end{equation*}
  where $A>0$ is some constant. We can bound the Hausdorff measure
  \begin{align*}
    \mathcal H^t(\widetilde R(\pi(\bj))) &\le A \sum_{k=N}^\infty
                                           \sum_{\abs{\bi}=k}\sigma_1(T_{\bi\bj|_{\ell_k}})\dots
                                           \sigma_{\floor{t}+1}(T_{\bi\bj|_{\ell_k}})^{-\floor{t}}\cdot
                                           \sigma_{\floor{t}+1}(T_{\bi\bj|_{\ell_k}})^{t} \\
                                         &= A \sum_{k=N}^\infty
                                           \sum_{\abs{\bi}=k}\phi^t(T_{\bi\bj|_{\ell_k}}).
  \end{align*}
  But $N\in\NN$ was arbitrary, so recalling~\eqref{eq:upperconv}, we
  can take $N\to\infty$ to see that
  $\mathcal H^t(\widetilde R(\pi(\bj)))=0$. This proves the lemma.
\end{proof}

The next lemma shows that if the shrinking targets are projected
cylinder sets whose lengths grow super-linearly, then the
corresponding recurring set is zero-dimensional.

\begin{lemma}\label{lem:superlinear}
  Suppose that $\lim_{k\to\infty}\ell_k/k = \infty$. Then
  $\dim_{\cH}\widetilde{R}(\pi(\bj),\ell_k)=0$.
\end{lemma}

\begin{proof}
  We show that $\mathcal{H}^\delta(\widetilde{R}(\pi(\bj),\ell_k))=0$
  for any $\delta>0$.

  Let $g(k) = \ell_k/k$. Then $g(k)\to\infty$ as $k\to\infty$, and
  \begin{equation*}
    \sum_{k\in\NN} \sum_{\abs{\bi}=k}\phi^\delta(T_{\bi\bj|_{\ell_k}}) \leq \sum_{k\in\NN} m^k (\sigma_+^\delta)^{k+\ell_k}
    \leq  \sum_{k\in\NN} \left(m\cdot \sigma_+^{\delta(1+g(k))}\right)^k<\infty
  \end{equation*}
  because for all $k$ large enough, the parenthetical quantity is
  $<1/2$ (say). The argument from the proof of Lemma~\ref{lem:upper}
  shows that $\cH^\delta(\widetilde R(\pi(\bj)))=0$, and proves the
  lemma.
\end{proof}

\subsection{Lower bound}

Fix numbers $s<t<s_0$. Let $(n_k)$ be a sequence of natural
numbers. Let $C= \bigcap_{k= 1}^\infty R(n_k)$, where

\begin{align*}
  R(n_k) &= \left\{\bi \in \mathcal{A}^\NN:
           \sigma^{n_k}(\bi)|_{\ell_{n_k}} =
           \bj\mid_{\ell_{n_k}}\right\}\\
         &= \left\{\bi \in \mathcal{A}^\NN: \sigma^{n_k}(\bi)\in B_{n_k}
           \right\}.
\end{align*}
Notice that $C(k)= R(n_k)$ are finite unions of compact cylinder sets
and thus compact, so that $C$ is compact as well. Furthermore,
$C\subset R= \limsup R(k)$, and we hope to prove that
$\dim_{\mathcal H} C\ge s$.

\begin{lemma}\label{lem:cylinderbound}
  Let $E_{\ba}$ and $\{\ell_k\}$ be as in Theorem~\ref{thm}. Fix
  $t<s_0$. Then there is a probability measure $\mu$ supported on $C$,
  and a constant $H$, satisfying
  \begin{equation}\label{eq:cylinderbound}
    \mu[\bq]\le H\phi^t(T_{\bq})
  \end{equation}
  for all finite words $\bq$.
\end{lemma}

\begin{remark*}
  A measure $\mu$ satisfying~\eqref{eq:cylinderbound} could possibly be found
  through a construction similar to one used
  in~\cite{RajalaVilppolainen}, or perhaps an adaptation
  of~\cite{Kaenmaki04}, and this may relieve us from having to assume
  quasi-mul\-ti\-pli\-ca\-tivity~\eqref{eq:cone}. However, quasi-mul\-ti\-pli\-ca\-tivity is also needed in the observation \eqref{eq:refpressure}. For this reason,
  we have opted to present a self-contained proof of the existence of
  $\mu$, even if our proof \emph{does} rely on~\eqref{eq:cone}.
\end{remark*}

\begin{proof}
  Let $(n_k)$ be a sequence of natural numbers to be fixed later, with
  the property that $n_k \geq n_{k-1} + \ell_{n_{k-1}}$ for all
  $k\in\NN$. Let
  \[
    \mu_k= \frac {\sum_{\bi\in
        A(k)}\phi^t(T_{\bi})\delta_{\bi}}{\sum_{\bi\in
        A(k)}\phi^t(T_{\bi})}
  \]
  where
  \[
    A(k)=\left\{\bi : \abs{\bi} = k \textrm{ and } \exists \bj\in
      \mathcal{A}^\NN \textrm{ with } \bi\bj\in C\right\}
  \]
  is the set of $\bi$ of length $k$ that can be extended to be
  elements of $C$ (as defined above). This way $\mu_{n_k}$ is
  supported on the set of points that recur to the shrinking targets
  at times $n_1, n_2, \dots, n_k$. Define the probability measure
  $\mu$ as the limit of a convergent subsequence of $\{\mu_k\}$.

  Let $\abs{\bq}=q$.  With condition~\eqref{eq:cone} we have that
  \begin{multline}\label{eq:muk}
    \mu_k[\bq] = \frac{\sum_{\substack{\bi\in A(k) \\
          \bi\mid_q = \bq}}\phi^t(T_{\bi})}{\sum_{\bi\in
        A(k)}\phi^t(T_{\bi})} \leq
    \frac{\phi^t(T_{\bq})\sum_{\substack{\bi\in A(k) \\ \bi\mid_q =
          \bq}} \phi^t(T_{\sigma^q(\bi)})}{D\sum_{\bi\in
        A(k)}\phi^t(T_{\bi\mid_q})\phi^t(T_{\sigma^q(\bi)})} \\
    = \frac{\phi^t(T_{\bq})\sum_{\substack{\bi\in A(k) \\ \bi\mid_q =
          \bq}} \phi^t(T_{\sigma^q(\bi)})}{D\sum_{\bi\in
        A(q)}\phi^t(T_{\bi})\sum_{\substack{\bi\in A(k) \\
          \bi\mid_q = \bq}}\phi^t(T_{\sigma^q(\bi)})} =
    \frac{\phi^t(T_{\bq})}{D\sum_{\bi\in A(q)}\phi^t(T_{\bi})}
  \end{multline}
  for all $k>q$. We will now prove that the sequence $\{n_k\}$ can be
  chosen so that there is some $\kappa>0$ such that
  \begin{equation*}
    \Sigma_q:=    \sum_{\bi\in
      A(q)}\phi^t(T_{\bi}) \geq \kappa
  \end{equation*}
  for all sufficiently large $q$. This will prove the lemma with
  $H=(D\kappa)^{-1}$.

  Clearly, $\Sigma_q$ decreases in intervals of the form
  $n_k \leq q\leq n_k + \ell_{n_k}$, so we first show that given a
  non-decreasing $\iota:\mathbb{N}\to\mathbb{R}$, we can choose
  $\{n_k\}$ so that $\Sigma_{n_k + \ell_{n_k}}\geq \iota(k)$ for all
  $k$.

  Recall that $P(s,\bj)$ is decreasing in $s$. Therefore, since
  $t<s_0$, there is some $\epsilon >0$ such that
  \begin{equation}\label{eq:exponential}
    \sum_{\abs{\bi}=k}\phi^{t}(T_{\bi\bj|_{\ell_k}}) \ge \exp(k\epsilon)
  \end{equation}
  for all $k$ sufficiently large. Hence there is a number $n_1$ such
  that
  \[
    \Sigma_{n_1 + \ell_{n_1}} =
    \sum_{\abs{\bi}=n_1}\phi^{t}(T_{\bi\bj|_{\ell_{n_1}}}) \geq
    \exp(n_1 \epsilon) \geq \iota(1).
  \]
  Proceeding inductively, assume we have already chosen
  $n_1, n_2, \dots, n_{k-1}$ such that
  $\Sigma_{n_{k-1} + \ell_{n_{k-1}}}\geq \iota(k-1)$. Then
  \begin{align}
    \Sigma_{n_k + \ell_{n_k}} &=\sum_{\abs{\bi_1}= n_1}\sum_{\abs{\bi_2}= n_2 - n_1 - \ell_{n_1}}\dots \sum_{\abs{\bi_k}= n_k - n_{k-1} - \ell_{n_{k-1}}} \phi^t(T_{\bi_1 \bj|_{\ell_{n_1}} \dots \bi_k\bj|_{\ell_{n_k}}}) \nonumber \\
                              &\overset{\eqref{eq:cone}}{\geq} \sum_{\abs{\bi_1}= n_1}\dots \sum_{\abs{\bi_k}= n_k - n_{k-1} - \ell_{n_{k-1}}} D \phi^t(T_{\bi_1 \bj|_{\ell_{n_1}} \dots \bi_{k-1}\bj|_{\ell_{n_{k-1}}}}) \phi^t(T_{\bi_k\bj|_{\ell_{n_k}}}) \nonumber \\
                              &= \Sigma_{n_{k-1} + \ell_{n_{k-1}}} D \sum_{\abs{\bi} = n_k -
                                n_{k-1} - \ell_{n_{k-1}}} \phi^t(T_{\bi
                                \bj|_{\ell_{n_k}}}). \nonumber \intertext{Now, by the induction
                                hypothesis,}
                              &\ge \iota(k-1) D \left(\sum_{\abs{\bi} = n_{k-1} + \ell_{n_{k-1}}}\phi^t (T_{\bi}) \right)^{-1}\sum_{\abs{\bi}= n_k}\phi^t(T_{\bi\bj|_{\ell_{n_k}}}) \nonumber \\
                              &\overset{\eqref{eq:exponential}}{\ge} \left[\iota(k-1) D
                                \left(\sum_{\abs{\bi} = n_{k-1} + \ell_{n_{k-1}}}\phi^t
                                (T_{\bi}) \right)^{-1}\right]
                                \exp(n_k\epsilon), \label{eq:lastline}
  \end{align}
  as long as we choose $n_k$ sufficiently large. Notice that the
  quantity in square brackets is a positive number only depending on
  our choices of $n_1, \dots, n_{k-1}$, so we can certainly choose
  $n_k$ large enough that~\eqref{eq:lastline} exceeds $\iota(k)$, as
  claimed.

  Now, suppose that we are in an interval of the form
  $n_k + \ell_{n_k} \leq q\leq n_{k+1}$, and
  $q- n_k - \ell_{n_k}= \ell$. Then~\eqref{eq:cone} implies that
  \begin{equation*}
    \Sigma_q \geq D\Sigma_{n_k + \ell_{n_k}}\sum_{\bi=\ell}\phi^t(T_{\bi}).
  \end{equation*}
  Since $t<s_0<\dim E_{\ba}$, the expression
  $\sum_{\bi=\ell}\phi^t(T_{\bi})$ is bounded below by a constant that
  is uniform over $\ell\in\NN$, which proves that there is some
  $\kappa >0$ such that $\Sigma_q \geq \kappa$ for $q$ sufficiently
  large, as wanted.

  Returning to~\eqref{eq:muk}, the lemma is proved with
  $H=(D\kappa)^{-1}$.
\end{proof}

\begin{lemma}\label{lem:energy}
  The image measure $\pi_*\mu$ given by Lemma \ref{lem:cylinderbound}
  has finite $s$-energy for almost all $\ba$.
\end{lemma}
\begin{proof}
  Using \cite[Lemma
  3.1]{Falconer88},~\cite[Proposition~3.1]{Solomyak98}, and Fubini's
  theorem, the problem inside any ball of radius $\rho$ reduces to the
  study of finiteness of
  \[
    I = \iint_{R\times
      R}\frac{d\mu(\bi)\,d\mu(\bk)}{\phi^s(T_{\bi\wedge\bk})}.
  \]
  Now, using \eqref{eq:cylinderbound}, the fact that $\mu$ is a
  probability measure, and $t>s$
  \begin{align*}
    I& \le \sum_{q=1}^\infty \sum _{\abs{\bq}=q}\sum_{i\neq k} \phi^s(T_{\bq})^{-1}\mu[\bq i]\mu[\bq k]\\
     & \le Hm\sum_{q=1}^\infty \sum _{\abs{\bq}=q}\sum_{k=1}^m \frac{\phi^t(T_{\bq})}{\phi^s(T_{\bq})}\mu[\bq k] \\
     &\le Hm \sum_{q=1}^\infty \sigma_+^{q(t-s)}<\infty.
  \end{align*}
  Since $\rho$ is arbitrary, the claim follow for almost all $\ba$.
\end{proof}

\begin{proof}[Proof of Theorem~\ref{thm}]
  Fix a center $\bj$ and a sequence $(\ell_k)$ of target sizes such
  that the limit
  $\lim_{k\to\infty}\frac{1}{k}\log \varphi^s(T_{\bj |_{\ell_k}})$
  exists. The almost-sure existence and uniqueness of the zero point
  $s_0$ of the pressure follow from Lemma \ref{lem:number exists}, and
  the estimate $\dim_{\mathcal H}\widetilde R(\pi(\bj))\le s_0$ from
  Lemma \ref{lem:upper}.

  For the lower bound, let $s<s_0$. By Lemma \ref{lem:energy} and the
  potential theoretic characterization of the Hausdorff dimension,
  since $\mu$ is supported on $C$, we have
  $\dim_{\mathcal H}\pi(C)\ge s$ almost surely. Approaching $s_0$
  along a sequence gives the lower bound for the dimension of
  $\pi(C)\subset\widetilde R(\pi(\bj))$.
\end{proof}
\begin{remark*}\label{rem:inj}
Notice that, the dimension of the recurring set is not affected by the set of points where $\pi$ is not an injection, as long as this set has dimension smaller than $s_0$. This is the case for large classes of self-affine sets, for a Lebesgue-positive proportion of translation vectors $\ba$, and in these cases there is the more geometric interpretation of \S\ref{sec:on fractal} for the result of Theorem \ref{thm}. 
\end{remark*}

\section{Some further observations}
\label{sec:examples}

It is classical to study the shrinking target problem for an invariant
measure. Unfortunately there is no canonical choice for a measure on a
self-affine set. There is, however, a natural family of measures,
called Gibbs measures, which we now define. In general, a Gibbs
measure for the potential $\phi^t$, or a {\bf Gibbs measure at $t$},
is a $\sigma$-invariant probability measure $\mu$ on the symbolic
space $\mathcal A^\mathbb N$ satisfying
\begin{equation}\label{eq:gibbs}
  C^{-1}\le \frac{\mu[\bq]}{\phi^t(T_{\bq})\exp(-\abs{\bq}\mathcal P(t))}\le C. 
\end{equation}
Unfortunately, as demonstrated by \cite[Example
6.4]{KaenmakiVilppolainen10}, Gibbs measures do not always
exist. However, under the assumption \eqref{eq:cone} they do and,
further, a Gibbs measure at the zero point of the pressure
$\mathcal P$ is a measure of maximal dimension, that is,
\[
  \dim_{\mathcal H}(\pi_*\mu) = \dim_{\mathcal H} E_{\ba}
\]
for almost all $\ba \in \mathbb R^{dm}$. These facts were proved in
\cite[Theorem 5.2]{KaenmakiVilppolainen10},
\cite[p. 8-9]{KaenmakiVilppolainen10} and \cite[Theorem
2.2]{KaenmakiVilppolainen08}. Also see \cite[Theorem
4.1]{KaenmakiReeve}. This means that the Gibbs measure is a natural
candidate as a measure to the size of the recurring set. Now we are in
the position to formulate and prove the following corollary.

\begin{corollary}[Corollary to Theorem 2.1 of \cite{ChernovKleinbock01}]\label{cor:Borel}
  Assume $E$ is a self-affine set as in Theorem \ref{thm}. Let
  $(\ell_k)\subset \mathbb N$ be an increasing sequence tending to
  $\infty$, and $\bj\in \mathcal A^{\mathbb N}$. Denote by $\mu$ the
  Gibbs measure at the point $t$ where $\mathcal P(t)=0$. Then the
  measure of the recurrence set $\pi_*\mu(\widetilde R(\pi(\bj)))$ is
  either $0$ or $1$ according to whether the sum
  \begin{equation}\label{eq:measure condition}
    \sum_{k=1}^\infty\mu([\bj_{\ell_k}])
  \end{equation}
  converges or diverges.
\end{corollary}
\begin{proof}
  First, the fact that the convergence of the sum~\eqref{eq:measure
    condition} implies $\pi_*\mu (\tilde R(\pi(\bj)))=0$ follows from
  the classical Borel--Cantelli lemma. This is the case since
  $\pi_*\mu (\widetilde R(\pi(\bj)))=\mu(R(\bj))$ and $\mu$ is
  $\sigma$-invariant.

  Now assume that the sum \eqref{eq:measure condition}
  diverges. Notice that while the definition of Gibbs measures in the
  symbolic space corresponding to self-affine sets does not
  necessarily coincide with the classical one, under the assumption
  \eqref{eq:cone} the measure $\mu$ satisfies the Facts 1-3 in Section
  3 of \cite{ChernovKleinbock01}. As these facts are all the knowledge
  on $\mu$ needed to prove Theorem 2.1 of \cite{ChernovKleinbock01},
  it follows that $\mu(R(\bj))=\pi_*\mu (\widetilde R(\pi(\bj)))$ has
  measure $1$.
\end{proof}

We finish with two further questions related to Theorem \ref{thm}.
\begin{itemize}
\item Will the claim of Theorem \ref{thm} remain true without
  condition~\eqref{eq:cone}?
\item What about recurrence under $F$ to geometric sets $B_k$, for
  example to balls in $\mathbb R^d$? What could be the counterpart of
  $s_0$? One has to be a little careful when stating a result of this
  type since, as pointed out above, in the overlapping cases $F$ is
  not even well-defined, and in most cases the corresponding symbolic
  problem becomes very hard to track. In our treatment it is possible
  that points fairly close to the target point $\pi(\bj)$ on the
  fractal set are not considered to be recurring, if they are far away
  from $\bj$ in the symbolic space. Taking this kind of shortcut
  becomes impossible in the geometric case.
\end{itemize}

\bibliographystyle{plain}

\bibliography{vaitbib}

\end{document}